\documentclass[a4paper,12pt]{amsart}

\usepackage{a4wide}
\usepackage[UKenglish]{babel}
\usepackage{bbm}
\usepackage{amsmath,amsfonts,amsthm}
\usepackage{graphicx}
\usepackage{enumerate}

\theoremstyle{remark}\newtheorem{remark}{Remark}[section]
\theoremstyle{plain}
	\newtheorem{theorem}[remark]{Theorem}
	
	\newtheorem{lemma}[remark]{Lemma}
	
\theoremstyle{definition}

\newcommand{\abs}[1]{\left\vert#1\right\vert}

\newcommand{\G}{\Gamma}
\newcommand{\cE}{\mathcal{E}}
\newcommand{\cF}{\mathcal{F}}
\newcommand{\cU}{\mathcal{U}}
\newcommand{\cJ}{\mathcal{J}}
\newcommand{\cI}{\mathcal{I}}
\newcommand{\cM}{\mathcal{M}}
\newcommand{\cS}{\mathcal{S}}

\newcommand{\cW}{\mathcal{W}}
\newcommand{\cV}{\mathcal{V}}
\newcommand{\cX}{\mathcal{X}}
\newcommand{\dual}[2]{\langle#1,\,#2\rangle}

\newcommand{\normBL}[1]{\left\Vert#1\right\Vert^\ast_{BL}}
\newcommand{\R}{{\mathbb R}}
\newcommand{\N}{{\mathbb N}}

\newcommand{\out}{\textrm{Out}}
\newcommand{\inc}{\textrm{Inc}}
\numberwithin{equation}{section}
\begin{document}

\title[Traffic optimization on networks]{A measure theoretic  approach to traffic flow optimization on networks}
\date{\today}

\author{Simone Cacace}
\address{Dipartimento  di Matematica e Fisica, Universit{\`a} degli Studi Roma Tre, L.go S.Leonardo Murialdo 1, 00146   Roma, Italy}
\email{\tt e-mail: cacace@mat.uniroma3.it}

\author{Fabio Camilli}
\address{Dipartimento di Scienze di Base e Applicate per l'Ingegneria, ``Sapienza'' Universit{\`a}  di Roma, Via Scarpa 16, 00161 Roma, Italy}
\email{camilli@sbai.uniroma1.it}

\author{Raul De Maio}
\address{Dipartimento di Scienze di Base e Applicate per l'Ingegneria, ``Sapienza'' Universit{\`a}  di Roma, Via Scarpa 16, 00161 Roma, Italy}
\email{raul.demaio@sbai.uniroma1.it}

\author{Andrea Tosin}
\address{Dipartimento di Scienze Matematiche  ``G. L. Lagrange'', Politecnico di Torino, Corso Duca degli Abruzzi 24, 10129 Torino, Italy}
\email{andrea.tosin@polito.it}

\subjclass[2010]{49J20, 35B37, 35R02, 49M07, 65M99}
\keywords{Network, transport equation, measure-valued solutions, transmission conditions, optimization}
\begin{abstract}
We consider  a class of optimal control problems for  measure-valued nonlinear transport equations describing traffic flow problems on networks.
The objective is to minimise/maximise macroscopic quantities, such as traffic volume or average speed, controlling few agents,
for example  smart traffic lights and automated cars. The  measure theoretic approach allows to study in a same setting local and nonlocal    drivers
interactions and to consider the control variables as additional measures interacting with the drivers distribution.
We also propose a gradient descent adjoint-based optimization method, obtained by deriving first-order optimality conditions for the control problem,
and we provide some numerical experiments in the case of smart traffic lights for a 2-1 junction.
\end{abstract}

\maketitle
\section{Introduction}
During the last years, the study   of vehicular and pedestrian traffic flow problems has become a very active   area and an opportunity of
information exchange between mathematical investigation  and applied research.
From a mathematical point of view, these phenomena have been largely studied due to their high complexity and the literature offers a broad variety of models devoted to their description in a wide range of scenarios, see \cite{Bellomo,Bressan,Garavello} for reviews.
On the other side, from an engineering point of view, it is important   to model, simulate, predict, control and optimize vehicular and pedestrian traffic in our society. These issues become more and more central with the fast technological progress and it is of particular interest to understand how  the latest technologies, such as   smart traffic lights, self-driving cars or big data, can be used to improve the quality of movement for drivers or pedestrians on road networks  and urban roads, see \cite{Cascetta,Work}.\par
In this paper we propose a model to simulate and optimize    traffic flow on networks based on the theory of measure-valued transport equations.
In this approach, the population is represented by a probability distribution which evolves according to a velocity field depending on the position of the other individuals. In this way   short and long range interaction mechanisms are readily taken into account into the dynamics of the problem.
Moreover the measure approach  easily catches the multi-scale nature of vehicular traffic, composed both by a continuous distribution of indistinguishable cars and  by some special individuals such as automated cars and traffic lights. With respect to other models considering transport equations with nonlocal interactions
(see \cite{Ambrosio,Canizo,Colombo}), the peculiarity of our  model is to be defined on a network, posing additional difficulties for  the interpretation  in a measure-theoretic sense of the transition conditions at the vertices. Existence, uniqueness and continuous dependence results for the corresponding  measure-valued transport equation were provided in \cite{CDTloc, CDTnonloc}.\par
In \cite{Bon_Butt, cavagnari1, cavagnari2}, the authors consider     optimal control problems for measure transport equations in the Euclidean space.
Relying on a similar approach, we consider a model where, besides the driver distributions,   the velocity field depends   also     on a external distribution which interacts with the original population in order to optimize, for example, traffic volume or   average speed    on the road network. As in \cite{Albi,Work}, our aim is to show that a small number of external agents can improve the global behavior of the   population and, indeed, the typical examples of control variables we consider are  smart traffic lights and automated cars. Since the external distribution  is described by a measure evolving  according to   an appropriate dynamics,  other control variables, such as    information about the behavior of the traffic  on the global network, can be considered.\par
The paper is organized as follows.  In Section \ref{S1} we  introduce  the control problem from a theoretical point of view: network structure, transport equation and  cost functional. Section \ref{S2} is  devoted to two examples of control problem: traffic lights and self-driving cars as controls for vehicular traffic. Section \ref{S3} focuses on numerical analysis for these problems: description and properties of the chosen scheme and numerical tests on some case studies. In the Appendix we report the proofs of some theoretical results contained in the previous sections.
\section{Problem Formulation and theoretical setting}\label{S1}
In this section we describe the main components of the traffic flow model, i.e. the structural components (roadway and priority
rules at the junctions), the dynamics of drivers motion (velocity, interaction  with other drivers, influence of the  structural components)
and   the  control   problem which has to be  solved in order optimize the traffic flow on the network.
\subsection{Structural components}
Traffic routes are mathematically described by a network $\G = (\cV, \cE)$ where $\cE = \{e_{1}, e_{2}, \hdots, e_{|\cE|}\}$ is the set of arcs/roads while the crossroads are represented by the set of the vertexes $\cV = \{V_1, \hdots, V_{|\cV|}\}$. The network is   oriented  and  we write  $e_k\to e_j$ and, respectively, $x\to y$ for $x,y\in\G$   to mean    that $e_k$ comes before $e_j$ and, respectively, $x$    before $y$ in the orientation of the network.
We assume that $\G$ is endowed with the minimum path distance $d_{\G}$ and each arc $e_j \in \cE$ is parametrised by a continuous bijective map $\pi_k: [0,L_k] \to e_k,$   $L_k \in (0, +\infty]$, which complies with the orientation of $\G$, i.e. if $V_i,V_j \in \cV$ are the vertexes of the arc    $e_k$ oriented from $V_i$ to $V_j$, then  $\pi_k(0) = V_i$ and $\pi_k(L_k)=V_j$ \\
For every $V_i \in \cV$, we denote by $\inc (V_i)$ the set of arcs in $\cE$ whose end  point is $V_i$ and by $\out(V_i)$ the set of arcs in $\cE$ whose starting point is $V_i$.  Then, we  divide the set of the vertexes respectively in the sets of sources, sinks and junctions
\begin{align*}
\begin{array}{c}
\cS = \{V_i \in \cV |\, \text{Inc}(V_i) = \emptyset\}, \\
 \cW = \{V_i \in \cV | \,\out(V_i) =\emptyset \}, \\
 \cJ = \{V_i \in \cV |\, \text{Out}(V_i)\not=\emptyset,\, \text{Inc}(V_i)\not=\emptyset \}.
\end{array}
\end{align*}
Since the velocity term   depends  on the distribution of the cars  on all the network, in order to simplify
the notations  we prefer to consider a network without sinks, i.e. the set $ \cW$ is empty and the terminal arcs always  have infinite length.
We also denote by $L_0$ the minimal length of the edges in $\cE$, i.e.
\begin{equation}\label{min_length}
L_0=\min_{e_j\in \cE} \ell(e_j).
\end{equation}

A convenient framework to study transport problems     is given by the measure theoretic one, since  it allows to consider in a same setting
macroscopic quantities such as  a continuous distribution of drivers and microscopic ones such as traffic lights and other elements of the model.
  We set $\G_T=\G\times [0,T]$ and we consider the metric space $(\G_T,d)$ where
$d((x,t),(y,s))=d_\G(x,t)+|t-s|$. For a function $\phi:\G_T\to\R$ we
define the norm
$$\|\phi\|_{BL} = \|\phi\|_{\infty} +\sup_{\substack{x,\,y\in\G_T \\ x\ne y}} \frac{|\phi(x) - \phi(y)|}{d(x,y)} ,$$
and we consider  the Banach space $BL(\G_T)$   of bounded and Lipschitz continuous functions equipped with norm  $\|\cdot\|_{BL}$.
Denoted by  $\cM(\G_T)$   the space of finite measure on $\G_T$, we define   a dual  norm on this space  by
\[
\|\mu\|_{BL}^* = \sup_{\substack{\phi \in BL(\G_T) \\ \|\phi\|_{BL}\leq 1}}|\dual{\mu}{\varphi}|.
\]
Similar notations and definitions are employed for the Banach space $\cM(\G)$ and $\cM([0,T])$.
In the following we will always consider measures in  $\cM^+(\G_T)$, the cone of positive measures  in $\cM(\G_T)$.
By the Disintegration Theorem,  we consider measures $\mu \in  \cM^+(\G_T)$ which can be  decomposed as
$$\mu(dxdt) = \mu_t(x)dt,$$
where $\mu_t \in \cM^+(\G)$ represents the distribution   at time $t \in [0,T]$. We remark that throughout the paper
we only consider measures without Cantorian part, since this kind of measure does not have any significant
interpretation for flow  traffic problems.
To model the   behavior  of   drivers at   junctions  we  assign a \emph{distribution matrix}
$P(t) = (p_{kj}(t))_{k,j = 1}^{|\cE|}$, for $t \in [0,T]$, satisfying the following properties
\begin{equation}\label{distr_matrix}
\begin{array}{c}
p_{kj} \in BV([0,T]),\quad p_{kj}(t) \in [0,1],\\[4pt]
\sum_{j=1}^{|\cE|}p_{kj}(t) = 1,\quad \forall t\in [0,T],\,\forall k=1,\dots,|\cE|, \\[4pt]
p_{kj}(t)=0\quad \text{if either $e_k\cap e_j=\emptyset$ or $e_j\to e_k$.}
\end{array}
\end{equation}
Here $p_{kj}(t)$ represents the fraction of drivers which at time $t$ flows from an arc $e_k$ to an arc $e_j$. Hence, for every arc $e_k$, we have a discrete probability distribution $P_k(t) = \{p_{kj}(t)\}_j$ which describes the behaviour of drivers at the junction at time $t$. This quantity is defined on the basis
of the   knowledge of the statistical behavior of the traffic at a given day time (see Gentile's work \cite{g1, g2}). The assumptions in \eqref{distr_matrix} implies the mass cannot concentrate at the vertexes and therefore the total mass is conserved at the internal junctions. Since we   consider  measures $\mu \in \cM^+(\G_T)$ without Cantorian part, we assume that $p_{kj} \in BV([0,T])$ so that for a measure $\mu\in\cM^+(\G_T)$ the product $p_{kj}\cdot \mu$ still has no Cantorian part.\par
\subsection{Driver motion}
We now describe the nonlinear  transport system which models the evolution of the
traffic  on the network. The   components of the system  are the differential equations governing
the evolution of the traffic inside the  arcs and the transition conditions at the vertices regulating the
distribution  of the traffic flow at the junctions. It is important to remark that the velocity term
is nonlocal since drivers usually have a local knowledge of the traffic distribution  in a visual
area in front of them; moreover they may have a global  knowledge
of the traffic distribution on the entire network thanks to appropriate navigation equipments.\par
We prescribe the initial mass distribution over $\Gamma$
$$m_0=\sum_{j\in J}m_0^j\in \cM^+(\G),$$
where $m_0^j$ is restriction of $m_0$ to   $e_j$, and the incoming traffic measure at the source nodes
$$\sigma_0 = \sum_{V_i \in \cS}\sigma^i_0, \quad \sigma^i_0 \in \cM^+([0,T]),$$
where $\sigma_0^i$ is the restriction of $\sigma_0$ to   $V_i$, representing the flow of cars
entering in the road network at the vertex $V_i$.
We consider the following system of measure-valued differential equations on $\G_T$  for the unknown  measure $m=\sum_{j\in J}m^j\in \cM^+(\G_T)$
\begin{equation}\label{system}
\left\{	\begin{array}{lr}
		\partial_{t}m^{j}+\partial_x(v^j[m_t,\mu_t]m^j)=0\quad &x\in e_j,\,t\in (0,\,T],\, j=1,\dots,|\cE|  \\[2mm]
		m_{t=0}^j=m_0^j   & x\in e_j,  j=1,\dots,|\cE| \\[2mm]
		m^{j}_{V_i=\pi_j(0)}=
			\begin{cases}
				\sum\limits_{e_k\in\inc(V_i)} p_{kj}(t)m_{V_i=\pi_k(1)}^k & \text{if } V_i\in\cI \\[3mm]
			\sigma_0^i & \text{if } V_i\in\cS,
			\end{cases}
 & j=1,\dots,|\cE|.
	\end{array}
\right.
\end{equation}
Observe that,  for each arc $e_j$, if the initial vertex $V_i=\pi_j(0)$ is internal, then  the boundary condition at $V_i$  is given by  a measure representing the mass flowing in $e_j$ from the arcs incident to the vertex according to the distribution  matrix  $P(t)$; if the initial vertex $V_i=\pi_j(0)$ is  incoming traffic vertex, the inflow measure is   the prescribed datum $\sigma_0^i$. The outflow measure, i.e. the part of the mass leaving the arc from the final vertex $V_k=\pi_j(1)$, is not given a priori but depends on the evolution of the measure $m$ inside the arc.\\
The velocity $v=(v^j)_{j=1}^{|\cE|}$ depends on the   solution $m_t $ itself,  as well as on  another distribution  $\mu_t\in \cM^+(\G)$, representing
external forces acting on the drivers such as  traffic lights and autonomous vehicles (more details will be given
in the next section where we   consider specific models). We assume that
\begin{itemize}
    \item[\textbf{(H1)}] $v$ is non-negative and bounded by  $V_{max}>0$;
    \item[\textbf{(H2)}] $v$ is Lipschitz with respect to the state variable, i.e. there exists $L > 0$ such that $\forall  x,y \in e_j$, $m_i,\,\mu_i \in \cM^+(\G)$, for $i = 1,2$
$$|v^j[m_1,\mu_1](x) - v^j[m_2,\mu_2](y)|\leq L(|x - y| + \|m_1 - \mu_1\|_{BL}^* + \|m_2 - \mu_2\|_{BL}^*); $$
\end{itemize}
For the definition of measure-valued  solution to the system \eqref{system}, we refer to \cite{CDTnonloc}.
The next theorem summarize the main results concerning existence, uniqueness and regularity of the measure-valued solution to \eqref{system} in case of a fixed $\mu \in \cM^+(\G)$.
\begin{theorem}\label{existencenet}
There exists a unique  $m\in\cM^+(\G_T)$ which is a measure-valued solution to \eqref{system}.
Moreover,
\begin{itemize}
\item[i)]
    Given   initial data  $m_{0,1}, m_{0,2}\in\cM^+(\Gamma)$    and  boundary data $\sigma_{0,1}, \sigma_{0,2}\in\cM^+( [0,\,T])$  and denoted by $m_1, m_2 \in \cM^+(\G_T)$ the corresponding solutions, there exists a constant $C=C(T)>0$ such that
    \[
    	\|m_{T,2}-m_{T,1}\|^{*}_{BL}\leq C\left(\normBL{m_{0,2}-m_{0,1}}+\normBL{\sigma_{0,2}-\sigma_{0,1}}\right).
    \]
\item[ii)]
    There exists a positive constant $C=C(T)$ such that
    \[
    	\normBL{m_t-m_{t'}}\leq C\abs{t-t'}+\sigma_0((t',\,t])
    \]
    for all $t',\,t\in [0,\,T]$ with $t'<t$.
\end{itemize}
\end{theorem}
We will   consider a velocity field  of the form
\begin{equation}\label{velocity}
 v[m,\mu](x) :=\max\{ v_{f}(x) - v_{i}[m](x)-v_{e}[\mu],0\}
\end{equation}
where $v_{f}: \G\to\R^+$ is the desired velocity   representing  the speed of a car over a free road, $v_{i}[m](x)$ is the interaction  term   due to the presence of  other cars  on the roads and   $v_{e}[\mu]$ is the interaction term with an external distribution $\mu$.
Here we describe the  velocities  $v_f$ and $v_i$, while in the next section we will consider velocities $v_{e}[\mu]$ corresponding to
the specific models considered.\\
Concerning the free flow speed $ v_{f}(x)$, which depends only on the state variable $x$, we assume that this
function is  positive, bounded and  Lipschitz continuous  on each arc  $e_j$ of the network $\G$. Hence $\textbf{(H1)-(H2)}$
are easily verified for $v_f$.\\
We  consider  a interaction velocity $v_i$ given by the   functional $$v_i[m](x) := \int_{\G}K(x,y)dm(y).$$
The  interaction kernel $K$ is defined as
\begin{equation}\label{kernel}
  K(x,y) = k(d_\G (x,y))\chi_{\mathcal D(x)}(y)
 \end{equation}
where  $k$ is  a  Lipschitz continuous, non increasing, bounded function  representing the strength of interaction among cars   in dependence on their distance   and $\chi_{\mathcal{D}(x)}$  is the characteristic  function of the set $\mathcal D(x)$ representing the visual field of the driver.
We assume that a driver has  only  the knowledge of the distribution of  cars on the  roads adjacent to the current position and therefore we define  the visual field as
\[\mathcal D(x)= \{y\in\Gamma: \,x\to y,\, d_\G(x,y)\leq R\}\]
with $R<L_0$  and $L_0$  defined in \eqref{min_length}.
Hence it follows that,  given  $x\in e_k$, if  $V_i=\pi_k(L_k)\in\cV$  we have  $\mathcal D(x)\subset e_k\cup( \bigcup_{e_j\in Out(V_i)}e_j)$.
We  prescribe    for any $e_j\in Out(V_i)$   a weight $\alpha_{kj}$  satisfying
\begin{align*}
0\le   \alpha_{kj}\le 1, \qquad \sum_{j=1}^{J}\alpha_{kj}=1,\\
\alpha_{kj}=0\quad \text{if either $e_k\cap e_j=\emptyset$ or   $e_j\to e_k$}.
\end{align*}
where the coefficients  $\alpha_{kj}$  represent the priority of a given route in the choice of the driver depending on the basis of the observed traffic distribution.
In conclusion,  the  interaction velocity at $x\in e_k$ is given
 \begin{equation*}
 v_i[m](x) = \sum_{e_j \in \cE}\alpha_{kj}\int_{\G}k(d_{\G}(x,y))\chi_{\mathcal D(x)\cap (e_k\cup e_j)}(y)dm(y).
 \end{equation*}
Since the function $K$ defined in \eqref{kernel} is nonnegative and  bounded, then
\begin{align*}
  0\le v_i[m](x) \leq C m(\G),&&\forall x \in \G,\\
  |v_i[m_1](x) - v_i[m_2](x)|  \le  C \|m_1 -  m_2\|_{BL}^*,&&\forall x \in \G,\,\forall m_1, m_2 \in \cM^+(\G),
\end{align*}
and therefore $\textbf{(H1)}$ and  $\textbf{(H2)}$ are satisfied. The Lipschitz continuity  with respect to $x$ is   more delicate  and
for its proof we refer to \cite[Sect.5]{CDTnonloc}. A specific example of function $k$ is given by
\[
   k(x,y)=\frac{\mu_2}{(\mu_1+ d_\G(x,y))^\beta}
\]
which  is  inspired by a typical Cucker-Smale nonlocal interaction (see \cite{CuSma}).

\subsection{Mobility optimization}
We introduce a class of optimization problems  on networks involving the distribution $m$, given by  the solution of \eqref{system}, the external distribution
$\mu$      and a   control variable $u$ which has  to be designed in order to minimize/maximize a  given objective functional.  \\
We assume that the set of the admissible controls is given by   a Banach space $(\cU,\|\cdot\|_\cU)$. We also  denote
by $\cM^+_{M}(\G_T)$ the  set of the measures   $\mu\in \cM^+(\G_T)$ such that $\|\mu\|_{BL}^* \le M$. Then the state space of  the control problem
is given by  the   space $(\cX,\|\cdot\|_\cX)$  where
\begin{align*}
\cX&=\cM^+_{M}(\G_T)\times \cM^+_{M}(\G_T) \times\cU,\\
\|\cdot\|_\cX&=\|\cdot\|_{BL}^* +\|\cdot\|_{BL}^*+ \|\cdot\|_{\cU}.
\end{align*}
For a given initial distribution $m_0\in \cM^+(\G)$ and an incoming traffic  distribution  $\sigma_0\in \cM^+([0,T])$, we  consider
the optimization problem
\begin{equation}\label{problem1}
 \left\{
 \begin{array}{l}
  \min\{J(m, \mu, u) : (m,\mu, u) \in \cX\},\\[4pt]
  \text { subject to the state equation \eqref{system}.}
  \end{array}
  \right.
\end{equation}
It is convenient to rewrite  the previous  minimization problem  in  the following  equivalent form
\begin{equation}\label{problem2}
\min\{J(m, \mu, u) + \mathbbm{1}_A(m, \mu, u): (m, \mu, u) \in \cX\},
\end{equation}
where $A := \{ (m, \mu, u) \in \cX;\,  m\, \text{solves}\,\eqref{system}\}$ and  $ \mathbbm{1}_A$ is the indicator function of the set $A$ defined as
$$
\mathbbm{1}_A(x) := \left\{\begin{array}{ll}0,& x \in A,\\+\infty&\text{otherwise.}\end{array}\right.
$$
A straightforward application of the direct method in Calculus of Variations gives the following existence result for
the minima of \eqref{problem2}.
\begin{theorem}\label{theo_min} Assume that
\begin{itemize}
\item $J:\cX \to \R\cup\{+\infty\} $ is bounded from below;
\item $J$ is lower semicontinuous in $\cX$, i.e. for any $(m_n,\mu_n, u_n) \subset \cX$  such that $(m_n,\mu_n, u_n) \to (m,\mu, u)$, it holds
$J(m,\mu, u)\leq \liminf_{n\to\infty}J(m_n,\mu_n, u_n);$
\item the set $A$ is closed under the topology induced by $\|\cdot\|_\cX.$
\end{itemize}
Then the minimization problem \eqref{problem1} has a solution.
\end{theorem}
A typical example of functional  to be minimized  is     of the form
\begin{equation}\label{functional}
J(m,\mu,u) := - \int_0^T\int_\G v[m_t, \mu_t]dm_t(y)dt + \int_{\G\times[0,T]}f(x,t,u)dm(x,t),
\end{equation}
where the first  term in \eqref{functional} represents the mean velocity on the network, while    the second  one
is a feedback term which depends on the choice of $f$. For example, if $f(t,x,u)= \chi_B(x)$, where $B \subset \G$ is closed, the functional minimizes the amount of mass $m_t$ in a closed region $B$ during the time interval $[0,T]$.
Another interesting class of control problems are minimum time control introduced, in a measure theoretic setting, in \cite{cavagnari1,cavagnari2}.
\section{Model examples: traffic lights and autonomous cars}\label{S2}
This section is devoted to    applications of the abstract setting previously described   with the discussion of two significative
problems in traffic flow optimization. In the first example,  we   optimize the duration of traffic lights in order  to improve the circulation on the road network; in the second example, we aim to regulate the traffic flow by a fleet of autonomous car. \\
 For both these models we assume that the control variable $u$ influences the traffic flow distribution $m$ only by means of  an    external distribution $\mu=\mu [u]$. Hence the functional to be minimized in \eqref{problem2} is  of the form $J(m,u)$ with $m$ subject to \eqref{system} and $\mu$ determined by another dynamical system for a given initial configuration $\mu_0$.

\subsection{Smart traffic lights}
An important element of a road network   model is given by  \emph{traffic lights}: they influence the behavior of the drivers near the junction  and can be used as an external control to regulate the traffic flow. To model a traffic light,  we follow the approach in \cite{Got_Ute}. Relying on the measure-theoretic setting, we describe a traffic light as a measure   $\theta\in \cM^+(\G_T)$, which is a Dirac measure in space and a densirty with bounded variation in time.\\
We assume that there is at most one  traffic light for each road and that it is located closed to the  terminal  vertex $V^i\in \cV$ of the  arc $e_j$. Since the position is fixed a priori while the activity changes in time, a traffic light can be represented, with an abuse of notation, as the measure
\begin{equation}\label{traffic}
\sum_{j\in \inc(V_i)} \int_{0}^T u_j(t) \delta_{V^i}(y)dt,
\end{equation}
where    $u_j \in BV([0,T], \{0,1\})$    is a function representing the state of the
traffic light:  $u_j(t) = 1$ if the light is red, $u_j(t) = 0$ if  green  (for simplicity, we do not  consider  a yellow phase since  the   corresponding driver reaction is strongly influenced by  drivers' culture). \par
Concerning the light  phases,  in order   to exclude  unrealistic scattering phenomena,
we fix two positive times $T^R_i, T^G_i>0$ and we assume that the red phase    cannot  last more then $T^R_i$ and, analogously, the green phase must last at least $T^G_i$ to guarantee a proper traffic flow.
Hence denoted by    $\tau_1, \tau_2 \in [0,T]$ two consecutive switching times     of the traffic light  on the arc $e_j$ (corresponding to jump discontinuities  of $u_j$),  we assume that
\begin{equation}\label{hyp_light1}
\begin{split}
   \text{ if $u_j(\tau^+_1) = 1$, then $|\tau_1 - \tau_2| < T^R_i,$}\\
   \text{if $u_j(\tau^+_1) = 0$, then $ |\tau_1 - \tau_2|> T^G_i$.}
   \end{split}
\end{equation}
Moreover we assume that a  traffic light can be green only for one of the incoming roads in a junction, i.e.  
\begin{equation}\label{hyp_light2}
\begin{array}{l}
\sum_{j \in \inc(V^i) }u_j +1 =N\\[8pt]
T^R_i\ge(N-1)T^G_i
\end{array}
\end{equation}
where  $N=\#(\text{Inc}(V_i))$.\\
Denote by $\cF\subset \cE$  the set of the arcs   containing  a traffic light. Recalling    \eqref{traffic}, we consider the    measure $ \mu(x,t)=\sum_{j\in J}  u_j(t) \mu^j(x,t)$ on $\G_T$ where
$\mu^j(x,t)\equiv 0$ if $e_j\not\in \cF$ and
$\mu^j (x,t)= \delta_{V^i}(x)  dt$  if $e_j \in \cF\cap \text{Inc}(V_i)$. The term $u_j$, the phase duration  of the traffic light on the road $e_j$,
can be interpreted as the control variable. The set of admissible controls is given by
 \begin{equation}\label{control_light}
 \cU = \{u= \{u_j \}_{j=1,\dots,|\cE|}: \, u_j \in BV([0,T], \{0,1\})\,  \text {and satisfies  \eqref{hyp_light1},  \eqref{hyp_light2} }\}
\end{equation}
To describe the interaction of the drivers with the traffic lights, we define  an   external velocity term $v_e[\mu]$ in \eqref{velocity}. Fixed an arc $e_j\in\cF\cap\text{Inc}(V_i)$,  then  the restriction of $v_e[\mu]$ to the arc $e_j$ is given by

$$
v^j_e[\mu](x) := \int_{\G}H(x,y)d\mu_t(y)=u_j(t)H(x,V_i)\delta_{e_j}(x).
$$
We  assume that the interaction kernel   $H$ is  given   by
\begin{equation}\label{ker_light}
H(x,y) =\left\{ \begin{array}{ll}
v_f\max\left\{\left(1-  \frac{  d_{\G}(x,y)}{R}\right),0 \right\},& \text{if $x\to y$, $d_{\G}(x,y)\leq R$},\\[6pt]
0&\text{otherwise},
\end{array}
\right.
\end{equation}
where $v_f$ is the desired velocity and  $R\le L_0$, for $L_0$ as in \eqref{min_length}, is the visibility radius.
The driver interaction  with the traffic light, tuned by the signal $u_j$, occurs only if the driver is sufficiently close  to the junction and   becomes stronger getting closer.\\
We need to show that the chosen set of control \eqref{control_light} satisfies the hypotheses of Theorem \ref{theo_min} for $\cX = \cM^+_{M}(\G_T)\times\cM^+_{M}(\G_T)\times \cU$.
\begin{lemma}\label{comp1}
The set of positive measures with bounded mass $\cM^+_{M}(\G_T)$ is compact with respect to $\|\cdot\|_{BL}^*$.
\end{lemma}
\begin{lemma}\label{comp2}
The set $\cU$ defined  in \eqref{control_light} is compact in $(BV^{|\cE|}([0,T]),\|\cdot\|_{L^1})$.
\end{lemma}
\begin{lemma}\label{closed}
Assume $\cX = \cM^+_{M}(\G_T)\times\cM^+_{M}(\G_T)\times \cU$, where $\cU$ satisfies the hypothesis of Lemma \ref{comp2}. The set $A$ is closed under the topology induced by $\|\cdot\|_{\cX}.$
\end{lemma}
The proofs of the previous results are given in Appendix.


\subsection{Regulating  traffic flow by means of autonomous cars}
In this second application, we aim to optimize the traffic flow by exploiting another distribution of cars, possibly given by autonomous vehicles,  of which we can control the velocity. Indeed some experiments (see Work \cite{Work}) have shown that it is possible to avoid stop-and-go phenomena regulating
the  interactions among drivers by means of external agents (autonomous vehicles, traffic light, signaling panels,etc.).
The approach in this section is inspired to   \cite{Bon_Butt}  where the authors present an optimization problem for a transport equation in  the euclidean space with the  control   represented by a second distribution $\mu$  evolving according to another  transport equation.\\
The dynamics of the autonomous cars is similar to the ones of rest of the driver, with the difference that it can be controlled
in order to minimize the objective functional. Hence for a given initial distribution $\mu_0$  (typically $\mu_0=\sum_{V_i\in \G_a}\delta_{V_i}$ for some finite set $\G_a\subset\G$), the measure $\mu\in\G_T$ representing the distribution of the  fleet of the autonomous car satisfies the nonlinear transport equation
\begin{equation}\label{google}
\left\{	\begin{array}{lr}
		\partial_{t}\mu^{j}+\partial_x(u\cdot v^j[m_t,\mu_t]\mu^j)=0\quad & x\in e_j,\,t\in (0,\,T],\, j=1,\dots,|\cE|  \\[2mm]
		\mu_{t=0}^j=\mu_0^j &   x\in e_j,  j=1,\dots,|\cE| \\[2mm]
		\mu^{j}_{V_i=\pi_j(0)}=
			\begin{cases}
				\sum\limits_{e_k\in\inc(V_i)} q_{kj}(t)\mu_{V_i=\pi_k(1)}^k & \text{if } V_i\in\cI \\[3mm]
			       0 & \text{if } V_i\in\cS,
			\end{cases}  &j=1,\dots,|\cE|
	\end{array}
\right.
\end{equation}
We assume that the velocity fields $v[m_t,\mu_t]$ in \eqref{google} is the same of problem \eqref{system}
and it is defined   as in \eqref{velocity}.
On the other side, since we want to regulate the velocity  of the  distribution $\mu$  we add a control   term $u$ and we assume that the control set is given by
\begin{equation}\label{control_autonomous}
 \cU = \text{Lip}_L(\G_T, [0,1]),
\end{equation}
i.e. the set of Lipschitz functions from $\G\times[0,T]$ to $[0,1]$ with Lipschitz constant   $L>0$. In this way,  if $v[m_t,\mu_t]$ satisfies the assumptions of Theorem \ref{existencenet}, then also $u\cdot   v[m_t,\mu_t]$ satisfies the same assumptions and therefore system \eqref{google}, given $(m_t)_{t \in [0,T]}$, admits a unique measure-valued solution. Moreover, since we require that  $u(x,t) \in [0,1]$, then the autonomous cars  can only slow the traffic distribution.
Observe that  system \eqref{google}  also differs from \eqref{system} for  the distribution matrix $Q=(q_{kj}(t))_{k,j = 1}^{|\cE|}$  at the junctions. Actually it is  reasonable to assume  that  $Q$ does not coincide with the distribution matrix $P$ since
the autonomous cars can behave differently from the rest of the drivers at the junctions. We assume that   the matrix $Q$ satisfies the assumptions in \eqref{distr_matrix}. Hence, the existence of solutions $(m, \mu)$ of the coupled transport system follows by a standard fixed point argument.\\

We conclude this section with the following Lemma:
\begin{lemma}\label{closed_2}
Assume $\cX = \cM^+_{M}(\G_T)\times\cM^+_{M}(\G_T)\times \cU$, where $\cU$ is defined by \eqref{control_autonomous}. The set $A$ is closed under the topology induced by $\|\cdot\|_{\cX}.$
\end{lemma}
This result can be proven as in the proof of Lemma \ref{closed}, using the Ascoli-Arzel\`a Theorem instead of Lemma \ref{comp2}.


\section{Numerical solution via optimality conditions}\label{S3}
In this section we formally derive first-order optimality conditions for the optimization problem \eqref{problem1} in the case of a traffic light for a 2-1 junction.
Then we build a gradient descent adjoint-based method to approximate the solution of the discretized optimality system and present some numerical experiments.

\subsection{Optimality conditions}
We  consider a network $\Gamma$ composed of a  junction with two roads converging in a single one, namely we have
$\cE=\{e_1,e_2,e_3\}$, $\cV=\{V_0,V_1,V_2,V_3\}$ and $\cJ=\{V_0\}$, $\cS=\{V_1, V_2\}$, $\cW=\{V_3\}$, $\text{Inc}(V_0) = \{e_1, e_2\}$ and
$\text{Out} (V_0)= \{e_3\}$, as shown in Figure \ref{fig:junction}.
\begin{figure}[!h]
\centering
\includegraphics[scale=0.5]{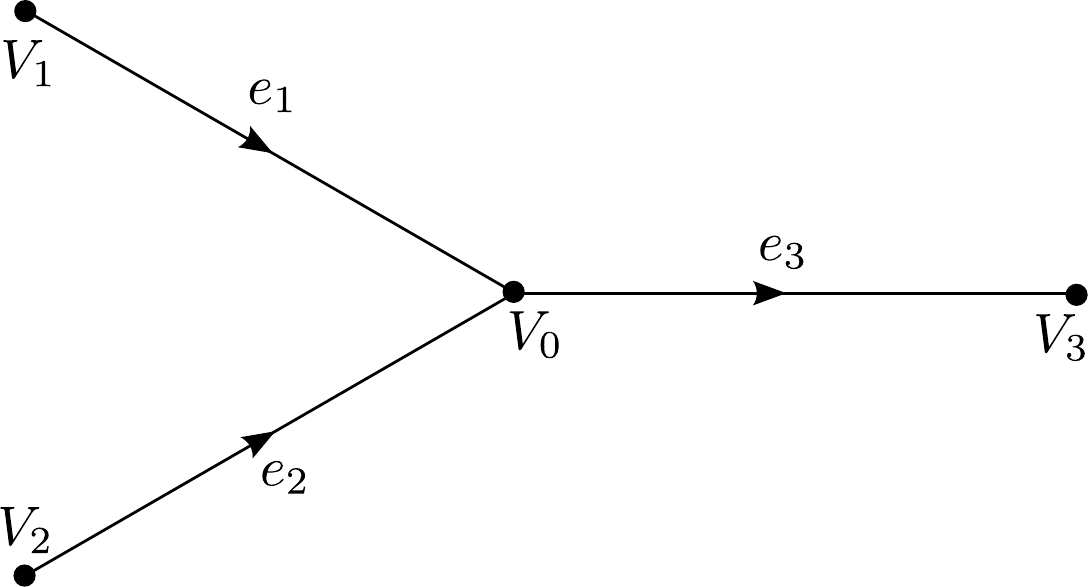}
\caption{Example of 2-1 junction}\label{fig:junction}
\end{figure}

To simplify the presentation, we neglect the drivers interaction term, since the computation in the general case is similar but more involved.
We place a traffic light  at $V_0$ in order to maximize the average speed on the network. In this setting a single control $u \in BV([0,T], \{0,1\})$
is enough to describe the system, indeed we define edge-wise the velocity $v$ by
$$v^1[u](x,t)=\max\{v_f^1(x)-u(t)H(x,V_0), 0\}\,,$$
$$v^2[u](x,t)=\max\{v_f^2(x)-(1-u(t))H(x,V_0), 0\}\,,$$
$$v^3(x,t)=v^3_f(x)\,,$$
where for $j=1,2,3$, $v_f^j$ is the free flow speed on $e_j$ and $H$ is defined as in \eqref{ker_light}.\\
Since the switching of the traffic light is intrinsically a discrete process, we translate the control problem into a finite dimensional setting.
More precisely, we consider a vector $s=(s_1,...,s_S)\in\R^{S}$, whose components represent
the durations of $S-1$ successive switches, where the integer number $S>1$ is fixed a priori. Then the control $u(t)$ is easily reconstructed from a given value $u(0)=u_0\in\{0,1\}$ at initial time
and from the switching times $\tau_i=\sum_{k=1}^i s_i$ for $i=1,...,S$. Defining recursively $u_i=1-u_{i-1}$ for $i=1,...,S$ and $\tau_0=0$ we set
(see Figure \ref{fig:control})
$$
u(t)=u^s(t)=\sum_{i=0}^{S-1} u_i\chi_{[\tau_i,\tau_{i+1})}(t)
$$
\begin{figure}[!h]
\centering
\includegraphics[scale=0.5]{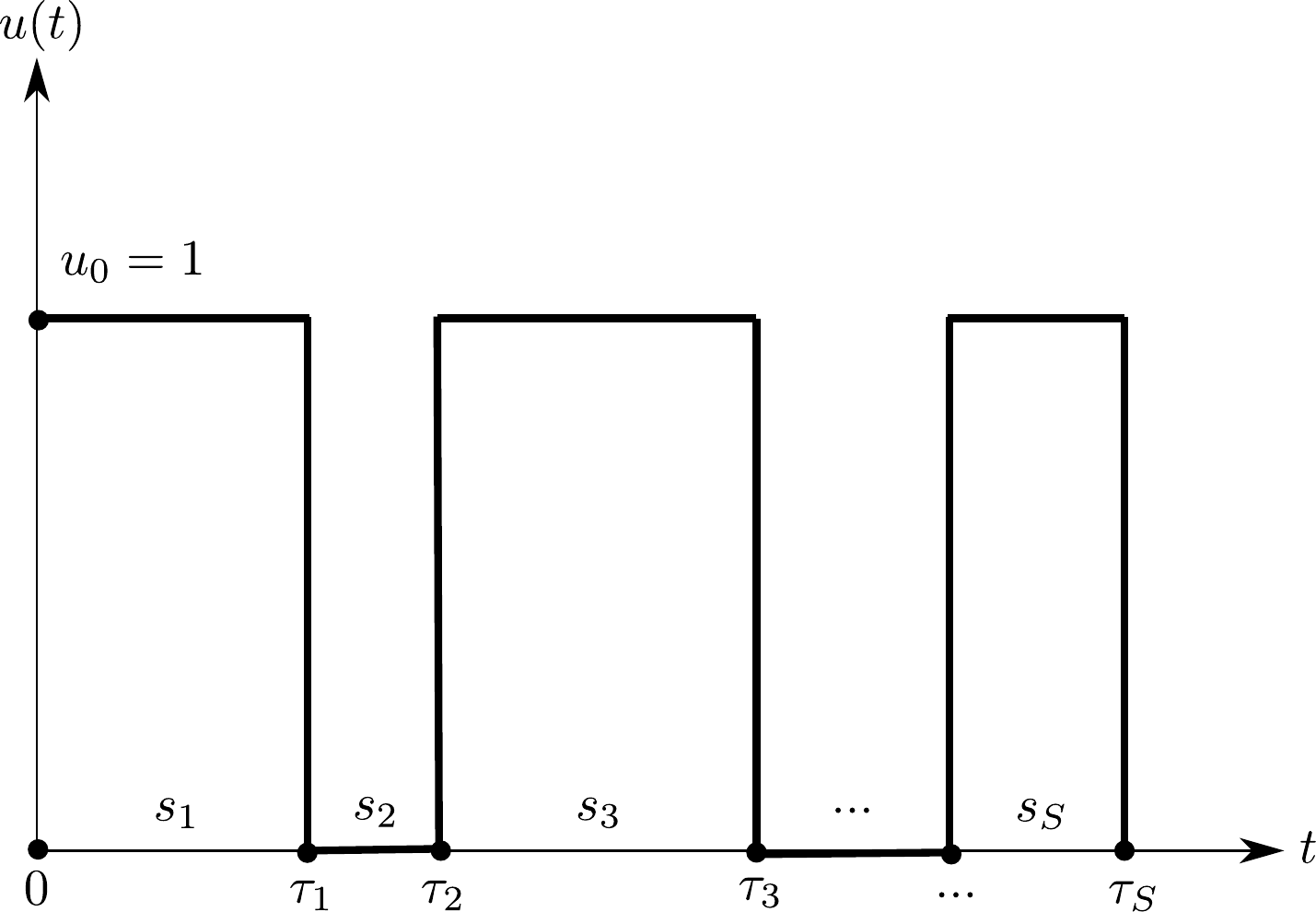}
\caption{Reconstruction of control $u$ from switching durations $s=(s_1,...,s_S)$}\label{fig:control}
\end{figure}

Following this approach we avoid several difficulties. Indeed, $BV([0,T], \{0,1\})$ is not even a vector space and taking admissible variations of a given control or
imposing constraints on the switching durations is in practice not easy at all.
One could work instead with the convex subset $BV([0,T];[0,1])$ of $L^2(0,T)$ and look for bang-bang controls.
This can prevent unrealistic mixing of mass at the junction, due to the additional yellow phase for the traffic light (intermediate values in $(0,1)$),
but chattering phenomena can occur.
In our setting we just work in $\R^S$, chattering is not allowed by construction, and we can easily apply variations/constraints to the switching durations
being sure that the control always remains in $BV([0,T], \{0,1\})$.\\
Assuming that the measure $m$ has a density, i.e. $dm=m(x,t)dx\,dt$ for some function $m:\Gamma\times [0,T]\to \R$,
we want to minimize the cost functional
\begin{equation}\label{mean_vel}
  J(m,u^s)= -\int_0^T\int_\Gamma v[u^s](x,t)m(x,t)\,dxdt\,,
\end{equation}
subject to
 \begin{equation}\label{system_ex}
\left\{
\begin{array}{ll}
 \partial_t m^j+\partial_x (v^jm^j)=0 & \mbox{in }e_j\times(0,T),\,j=1,2,3\\
 m^j(\cdot,0)=m_0^j & \mbox{in } e_j
\end{array}
\right.\end{equation}
We also assume null incoming traffic in the network during the whole evolution, imposing
\begin{equation}\label{boundary}
m^1_{x=V_1}=0,\quad m^2_{x=V_2}=0,\quad t\in[0,T]\,,
\end{equation}
and the mass conservation condition at the internal vertex $V_0$
\begin{equation}\label{transition}
m^3_{x=V_0} = m^1_{x=V_0} + m^2_{x=V_0}.
\end{equation}
We formally apply the  method of Lagrange multipliers in order to derive first-order optimality conditions. We define the
Lagrangian as
\begin{align*}L(m,u^s,\lambda) &:=J(m,u^s)+\int_0^T\int_{\Gamma}(-\partial_t \lambda-v\partial_x\lambda)m\,dxdt\\
&+ \int_{\Gamma}(\lambda(x,T)m(x,T)-\lambda(x,0)m_0(x))\,dx\\
 &+\sum_{j=1,2,3}\int_0^T(\lambda^j(V_j^{E},t)v^j(V_j^{E},t)m^j(V_j^{E},t)-\lambda^j(V_j^{I},t)v^j(V_j^{I},t)m^j(V_j^{I},t))\,dt\,,
\end{align*}
where  $V_j^I$ and $V_j^E$ denote  the initial and, respectively,  the final vertex  of the arc  $e_j$. Observe that the terms involving the Lagrange multiplier $\lambda$ derive from the weak formulation of the transport equation on $\Gamma$.\\
We evaluate the derivates of the Lagrangian  with respect to $m$ and $s$ (recall that $u=u^s$). We first consider an admissible increment    $w$ for $m$ which preserves
the boundary and transition conditions, i.e.
 \begin{equation}\label{admissible}
 w^1(V_1,t)=0\,,\qquad w^2(V_2,t)=0\,,\qquad w^3(V_0,t)=w^1(V_0,t)+w^2(V_0,t)\qquad t\in[0,T]\,,
 \end{equation}
 and we compute
\begin{equation}\label{lagrangian}
\begin{split}
 \langle \partial_m L,w\rangle&=\int_0^T\int_{\Gamma}(-\partial_t \lambda-v\partial_x\lambda -v)w\,dxdt+\int_{\Gamma}\lambda(x,T)w(x,T)\,dx\\
 &+\int_0^T\sum_{j=1,2,3}(\lambda^j(V_j^{E},t)v^j(V_j^{E},t)w^j(V_j^{E},t)-\lambda^j(V_j^{I},t)v^j(V_j^{I},t)w^j(V_j^{I},t))\,dt\,.
 \end{split}
\end{equation}
Imposing $\langle\partial_m L,w\rangle=0$ for any admissible  $w$, we get the following time-backward advection equation with a source term
 \begin{equation}\label{eq_lambda}
 -\partial_t \lambda^j-v^j\partial_x\lambda^j =v^j\qquad \mbox{in}\quad e_j\times(0,T),\,j=1,2,3,
 \end{equation}
 and  the final condition
 $$\lambda^j(x,T)=0\qquad \mbox{in}\quad e_j,\,j=1,2,3.$$
Note that for   \eqref{eq_lambda},   $V_3$  is an inflow vertex where  a boundary condition has to prescribed, while $V_1$ and $V_2$ are outflow ones.
Writing explicitly the remaining boundary terms in \eqref{lagrangian}, we have
\begin{align*}
\int_0^T(\lambda^1v^1w^1(V_0,t)-\lambda^1v^1w^1(V_1,t)+\lambda^2v^2w^2(V_0,t)\\
-\lambda^2v^2w^2(V_2,t)+\lambda^3v^3w^3(V_3,t)-\lambda^3v^3w^3(V_0,t))\,dt=0\,.
\end{align*}
By taking    $w$  compactly supported  in a neighborhood of $V_3$, we get the boundary condition
$$\lambda^3(V_3,t)=0\qquad \mbox{in}\quad [0,T]\,,$$
whereas for $w$ compactly supported  in a neighborhood of   $V_0$, recalling \eqref{admissible}, we get
\begin{equation}\label{boundary_terms}
 \int_0^T\{(\lambda^1v^1-\lambda^3v^3)w^1(V_0,t)+(\lambda^2v^2-\lambda^3v^3)w^2(V_0,t)\}\,dt=0\,.
\end{equation}
 The mass conservation condition  \eqref{transition} can be rewritten as
\[
   v^3(V_0,t)m^3(V_0,t)=v^1(V_0,t) m^1(V_0,t)+v^2(V_0,t)m^2(V_0,t)\qquad t\in[0,T]\,,
\]
since  the control law $u$ models  a traffic light which  bring to halt the speed of the drivers  at $V_0$  in $e_1$  and, alternatively,  in $e_2$,  in such a way that there is mass flow either from $e_1$ to $e_3$ or from  $e_2$   to $e_3$.
If $I_1\subseteq[0,T]$ is  an interval where $u(t)=1$ (red light for $e_1$), then in this interval the speed $v^1(V_0,t)$ is null and therefore
 $m^1(V_0,t)=0$   (recall that  mass concentration at the vertices is not admitted).
Similarly if $u(t)=0$ for $t\in I_2$ (red light for $e_2$),  we get  $m^2(V_0,t)=0$ for $t\in I_2$.
An admissible increment, in order to preserve the transition condition for $m$, has to satisfy the same property and by \eqref{boundary_terms}
we get
$$\lambda^3(V_0,t)v^3(V_0,t)= \lambda^1(V_0,t)v^1(V_0,t)+\lambda^2(V_0,t)v^2(V_0,t),$$
or, more explicitly,
$$\lambda^1(V_0,t)v^1(V_0,t)=\lambda^3(V_0,t)v^3(V_0,t) \qquad\mbox{if}\quad t\in\{v^1(V_0,t)\neq 0\}\,,$$
$$\lambda^2(V_0,t)v^2(V_0,t)=\lambda^3(V_0,t)v^3(V_0,t) \qquad\mbox{if}\quad t\in\{v^2(V_0,t)\neq 0\}\,.$$

We now compute  the derivative of $L$ with respect to $u^s$ for an increment $\varphi\in\R^S$
\begin{align*}
\langle \partial_s L,\varphi\rangle&=-\int_0^T\int_{\Gamma}\partial_s v\cdot \varphi(\partial_x\lambda+1)m\,dx dt
+\int_0^T\{\sum_{j=1,2,3}\lambda^j(V_j^{E},t)\partial_s v^j(V_j^{E},t)\cdot \varphi\, m^j(V_j^{E},t)\\
&-\lambda^j(V_j^{I},t)\partial_s v^j(V_j^{I},t)\cdot \varphi\, m^j(V_j^{I},t)\}\,dt\,.
\end{align*}
Recalling \eqref{boundary} and since $v^3$ is independent of $u^s$, we get
\begin{align*}
\langle \partial_s L,\varphi\rangle=\int_0^T\Big\{-\int_{e_1}\partial_s v^1\cdot\varphi(\partial_x\lambda^1+1)m^1\,dx
-\int_{e_2}\partial_s v^2\cdot\varphi(\partial_x\lambda^2+1)m^2\,dx\\
 +\lambda^1(V_0,t)\partial_s v^1(V_0,t)\cdot \varphi\,m^1(V_0,t)+\lambda^2(V_0,t)\partial_s v^2(V_0,t)\cdot\varphi\, m^2(V_0,t)\Big\}\,dt,
\end{align*}
where
$$\partial_s v^1(x,t)\cdot\varphi=-H(x,V_0)\nabla_s u^s(t)\cdot \varphi\,,\qquad \partial_s v^2(x,t)\cdot\varphi=H(x,V_0)\nabla_s u^s(t)\cdot \varphi$$
and
$$
\nabla_s u^s(t)\cdot \varphi=\sum_{i=1}^S(-1)^{u_{i-1}}\delta_{\tau_i}(t)\varphi_i\,.
$$
We conclude
\begin{align*}
\langle \partial_s L,\varphi\rangle=\sum_{i=1}^S(-1)^{u_{i-1}}\Big\{\int_{e_1}H(x,V_0)(\partial_x\lambda^1(x,\tau_i)+1)m^1(x,\tau_i)\,dx
-\lambda^1(V_0,\tau_i)H(V_0,V_0)m^1(V_0,\tau_i)\\-\int_{e_2}H(x,V_0)(\partial_x\lambda^2(x,\tau_i)+1)m^2(x,\tau_i)\,dx
 +\lambda^2(V_0,\tau_i)H(V_0,V_0) m^2(V_0,\tau_i)\Big\}\,\varphi_i\,.
\end{align*}
Summarizing, the dual problem for \eqref{system_ex}-\eqref{boundary}-\eqref{transition} is
$$
\left\{
\begin{array}{ll}
 -\partial_t \lambda^j-v^j\partial_x\lambda^j=v^j & \mbox{in }e_j\times(0,T),\,j=1,2,3,\\
 \lambda^j(\cdot,T)=0 & \mbox{in } e_j,\\
 \end{array}
\right.
$$
with the boundary condition
\[\lambda^3(V_3,t)=0,\quad \mbox{in }[0,T],\]
and the transmission condition
\[ \lambda^j(V_0,t)v^j(V_0,t)=\lambda^3(V_0,t)v^3(V_0,t) \quad\mbox{if }t\in\{v^j\neq 0\},\,j=1,2\,.\]
Finally, if we impose box constraints $T^G<s_i<T^R$ for $i=1,...,S$, the optimal solution $(m,u^s,\lambda)$ should satisfy,
for all $\bar s\in\R^S$ such that $T^G<\bar s_i<T^R$, the variational inequality
\begin{equation}\label{var_in}
\langle \partial_s L(m,u^s,\lambda),\bar s-s\rangle\ge 0.
\end{equation}\medskip
\begin{remark}
If the velocity field contains the drivers interaction term, then the dual problem for \eqref{system_ex}-\eqref{boundary}-\eqref{transition} is
given by
$$
\left\{
\begin{array}{ll}
 -\partial_t \lambda^j-v^j\partial_x\lambda^j- \nu*(m\partial_x\lambda) =v^j+\nu*m& \mbox{in }e_j\times(0,T),\,j=1,2,3\\
 \lambda^j(\cdot,T)=0 & \mbox{in } e_j\\
 \end{array}
\right.
$$
with the same boundary and transition conditions, where $(\nu*\phi)(x)=\int_\Gamma K(y,x)\phi(y)dy$. The additional terms in the equation
represent a time-backward counterpart of the nonlocal term in the forward equation. Indeed, note that the kernel $K$ is not symmetric by definition and the integration
is here performed with respect to the first variable, looking at $y\to x$ and not $x\to y$ as in \eqref{kernel} .
\end{remark}

\subsection{Discretization}
The above optimality system can be discretized using, for instance, finite difference schemes and solved by some root-finding algorithm.
Here we do not solve the whole discrete system at once, we instead obtain an approximate solution splitting the problem in three simple steps.
With a fixed control, we first solve the forward equation in $m$, then we solve the backward equation in $\lambda$, and finally update the control using
the expression we obtained for the gradient $\partial_s L$, iterating up to convergence. The resulting procedure is a gradient descent method, summarized in
the following algorithm.\\

\noindent{\bf Algorithm }[Forward-Backward system with Gradient Descent]
\begin{enumerate}
 \item[Step 0.] Choose $\varepsilon>0$, $\beta>0$ and set $J^{(0)}=0$;
 \item[Step 1.] Fix an initial guess for $s^{(0)}\in\R^S$, $u_0\in\{0,1\}$  and set $k=0$;
 \item[Step 2.] Use $s^{(k)}$ to build the control $u^{(k)}$;
 \item[Step 3.] Solve the forward problem for $m^{(k)}$ with control $u^{(k)}$;
 \item[Step 4.] Solve the backward problem for  $\lambda^{(k)}$ with control $u^{(k)}$;
 \item[Step 5.] Compute $J^{(k+1)}=J(m^{(k)},s^{(k)})$.\\ If $|J^{(k+1)}-J^{(k)}|<\varepsilon$ go to Step 8, otherwise update $J^{(k)}\leftarrow J^{(k+1)}$ and continue;
 \item[Step 6.] Compute $\partial_s L$ at $(m^{(k)},u^{(k)},\lambda^{(k)})$;
 \item[Step 7.] Update $s^{(k)}\leftarrow \Pi_{\{T^G,T^R\}}\left(s^{(k)}-\beta \partial_s L(m^{(k)},u^{(k)},\lambda^{(k)})\right)$, $k\leftarrow  k+1$ and go to Step 2 \\
 ($\Pi_{\{T^G,T^R\}}$ denotes the component-wise projection on the interval $[T^G,T^R]$);
 \item[Step 8.] Accept $(m^{(k)},u^{(k)},\lambda^{(k)})$ as an approximate solution of the optimal control problem for \eqref{mean_vel}.
\end{enumerate}

In the actual implementation of the algorithm, we employ a standard scheme for conservation laws with a superbee flux limiter, to solve the forward equation in $m$.
On the other hand, the adjoint advection equation in $\lambda$ is solved by means of a standard time-backward upwind scheme. We choose the numerical grid in space and
time subject to a sharp CFL condition, in order to mitigate the numerical diffusion and better observe the nonlocal interactions.
Moreover, we compute all the integrals appearing in the functional $J$, in the nonlocal terms and in the expression of the gradient $\partial_s L$,
by means of a rectangular quadrature rule. We also employ a simple inexact line search technique to compute a suitable step $\beta$  for the gradient update in Step 7.
Finally, the application of control constraints is easily obtained by projection. More precisely, given compatible durations
$0<T^G<T^R$ and the updated $s^{(k)}$ in Step 7, we set $s^{(k)}_i\leftarrow \max\{T^G,\min\{s^{(k)}_i,T^R\}\}$ for $i=1,...,S$.

\subsection{Numerical experiments}
As a preliminary test we compare the local and the nonlocal case. We consider only the evolution of the density $m$ along the edge $e_1$ and we set the control $u(t)\equiv 1$
to keep the traffic light at the end of the road activated (red) during the whole simulation. We choose the length $\ell(e_1)=1$ and $R_1=\frac18$ for the visibility radius of the traffic light.
On the other hand, we choose the nonlocal interaction kernel \eqref{kernel} with $k(r)=\frac{25}{1+r}$ and  visibility radius $R=15dx$, where $dx$ is the step size of the space grid.
Finally, we set the free flow speed $v_f^1\equiv 1$ and the initial distribution $m_0(x)=\chi_{[0.1,0.15]}(x)$.
Figure \ref{fig:locvsnloc} shows the evolution of $m$ and $v$ at different times. Top panels refer to the local case, bottom panels to the nonlocal one.
We represent the density $m$ in black and the velocity $v$ in red, decreasing from $v_f^1$ to zero with a linear ramp while approaching the traffic light, according to the definition
\eqref{ker_light} for $H$.
\begin{figure}[!h]
\centering
\includegraphics[scale=1]{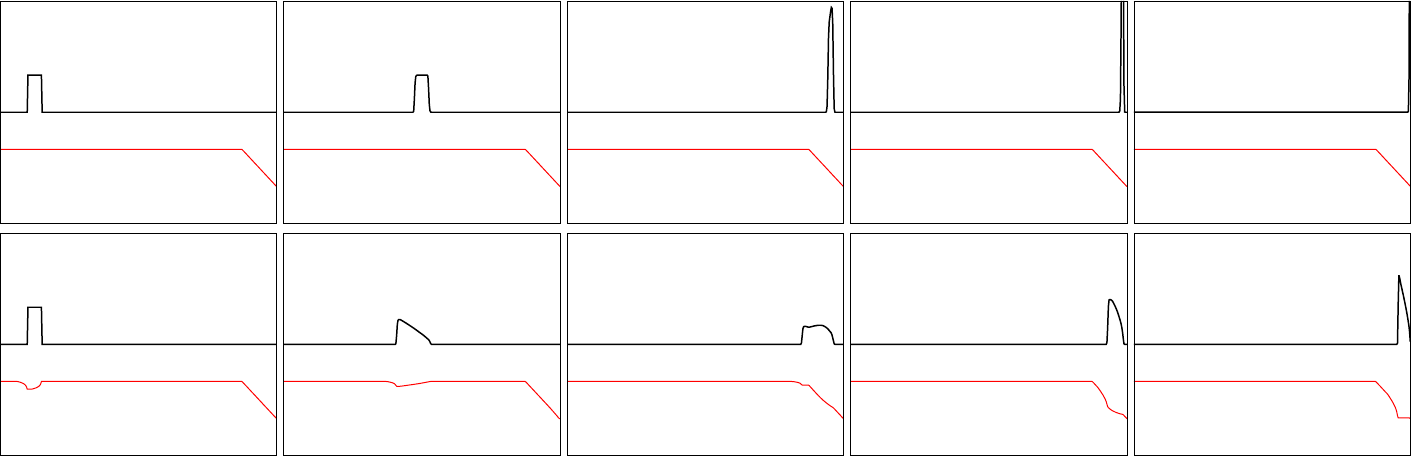}
\caption{Red traffic light: local case (top panels) vs nonlocal case (bottom panels)}\label{fig:locvsnloc}
\end{figure}\\
In the local case $v$ does not depend on time, since $u$ is constant. The density $m$ proceeds without changing profile (except some numerical diffusion at the boundary of its support),
then starts concentrating close to the traffic light. At the final time, all the mass is concentrated at the point closest to the traffic light.\\
In the nonlocal case, drivers interactions are clearly visible both in $m$ and $v$. The initial density
readily activates the nonlocal term in $v$, and $m$ starts assuming the well known triangle-shaped profile.
Close to the traffic light we observe a slowing-down, that propagates backward up to the beginning of the queue, preventing mass concentration.
At final time the profile becomes stationary, we observe that $v$ is zero in the whole support of $m$. \\

We proceed with a test for validating the proposed numerical method. We consider the case of a single switching time $\tau\in[0,T]$, namely we choose
$s=(s_1,s_2)=(\tau,T-\tau)$ without constraints and $u_0=1$, so that the corresponding control is just $u^s(t)=\chi_{[0,\tau]}(t)$ (red light on $e_1$ for $t\le \tau$). This reduces the optimization problem to a minimization in dimension one,
that can be analyzed by an exaustive search in $\tau$ and then compared with our adjoint-based algorithm. We set all the parameters as in the previous test,
in particular we choose constant free flow speeds $v_1^f=v_2^f=v_3^f\equiv 1$ and set $T=1.25$.
We also assume that, apart from $m_0$, no additional mass enters or leaves the network for all $t\in[0,T]$.

We start with $m_0=(m_0^1,m_0^2,m_0^3)=(\chi_{[0.1,0.15]}(x),\chi_{[0.6,0.65]}(x),0)$, i.e. two distributions of equal mass on $e_1$ and $e_2$ that arrive at the traffic
light at different times ($m_2$ first and then $m_1$). In Figure \ref{fig:vmeannormalized}(a) we plot the corresponding (normalized) mean velocity $\bar v(\tau)=-J(m,u^s)/M$ as a function of
$\tau$, where $M=\int_0^T\int_\Gamma m(x,t)dx\,dt$.
\begin{figure}[!h]
\centering
\begin{tabular}{cc}
\includegraphics[scale=.3]{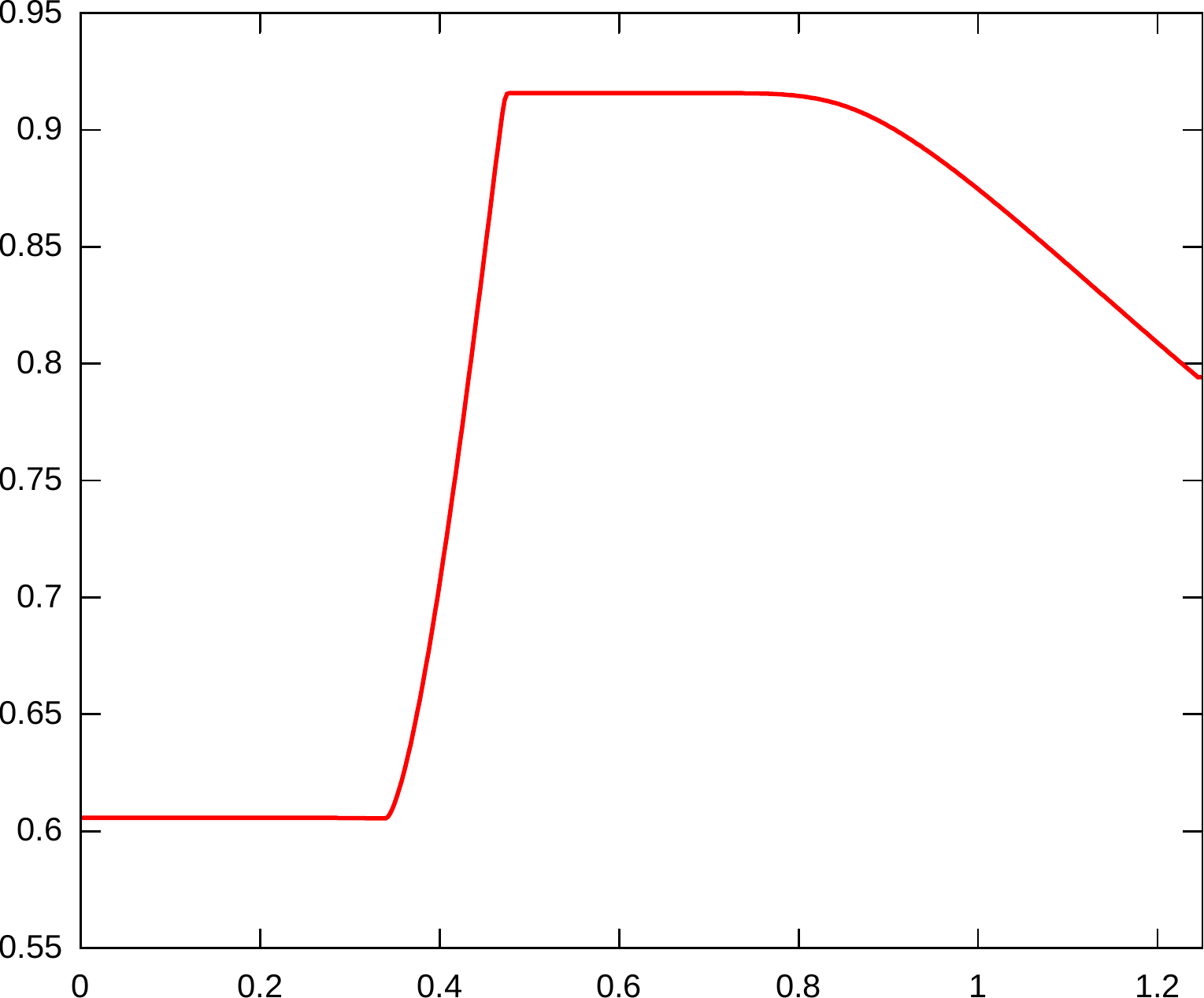}&
\includegraphics[scale=.3]{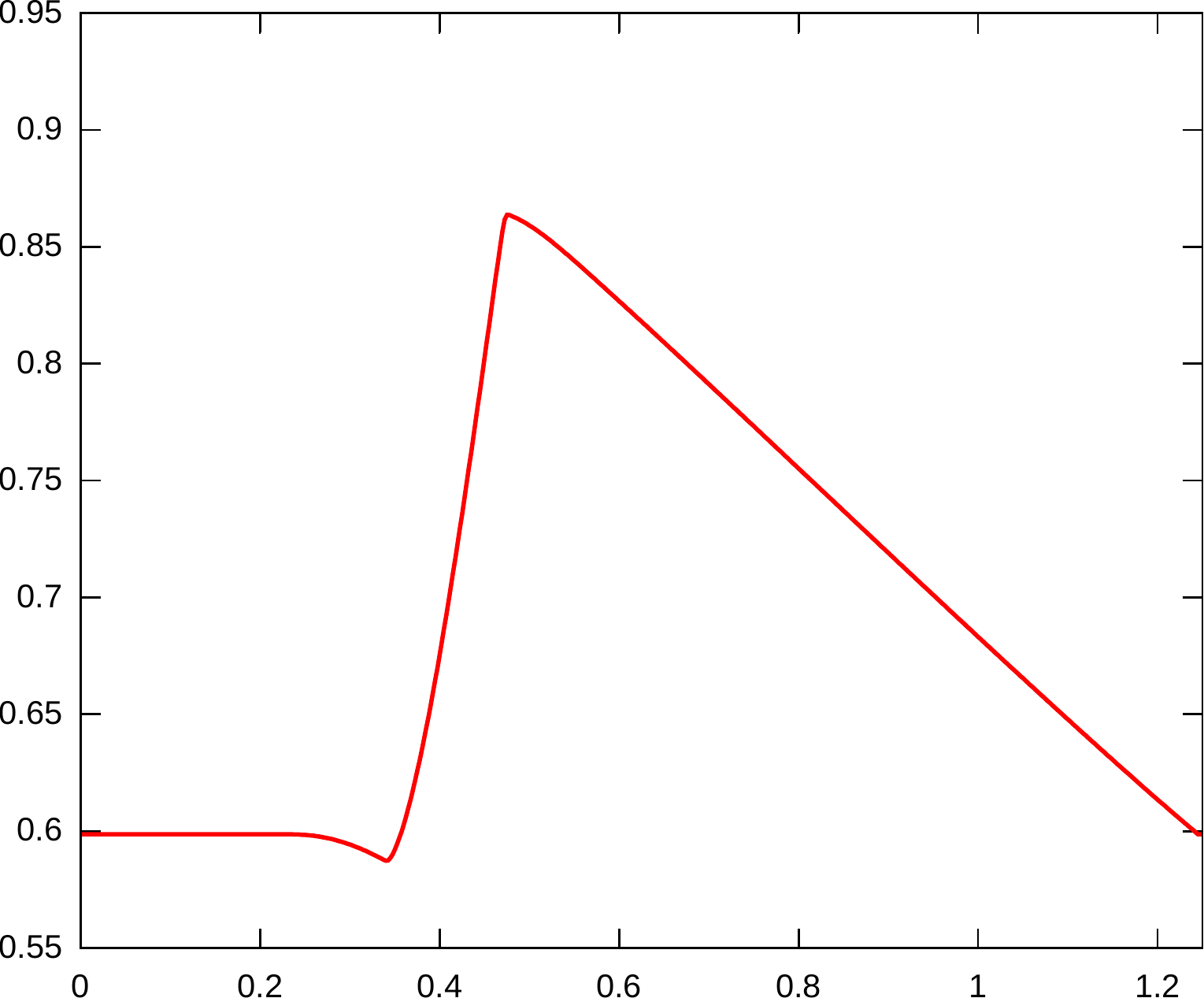}\\
(a)&(b)
\end{tabular}
\caption{Mean velocity for a single switch of the traffic light: well separated $m_1$, $m_2$ (a),  overlapping $m_1$, $m_2$ (b)}\label{fig:vmeannormalized}
\end{figure}
The scenario is pretty clear. If the switch occurs before $m_2$ reaches the traffic light, then only $m_1$ will move from $e_1$ to $e_3$ and the mean velocity cannot improve.
For larger values of $\tau$, also $m_2$ will gradually move to $e_3$, and $\bar v(\tau)$ increases. If now the switch is placed just after $m_2$ leaves $e_2$ and before
$m_1$ approaches the traffic light, we get the best performance, both distributions move as they are on a free road.
Note that, due to the nonlocal interactions, the maximum of $\bar v$ is less than the free flow speed.
Finally, as $\tau$ keeps increasing up to $T$, $m_1$ starts
getting stuck at the traffic light, and $\bar v(\tau)$ decreases.

Now let us repeat the exaustive computation of the mean velocity $\bar v(\tau)$ with $m_0=(m_0^1,m_0^2,m_0^3)=(\chi_{[0.6,0.65]}(x),\chi_{[0.6,0.65]}(x),0)$,
two distributions of equal mass on $e_1$ and $e_2$, starting at the same distance from the traffic light.
Figure \ref{fig:vmeannormalized}(b) shows the shape of the corresponding
$\bar v$. We observe that the maximum of $\bar v$ is lower than in the previous test, and it is achieved at a single point instead of an interval.
This clearly depends on the fact that the two densities are not well separated as before and it is not possible to place a switch without penalizing the overall traffic flow.
Moreover, note that an absolute minimum appears just after the initial plateau. Interestingly, this means that if the switch occurs too early both densities slowdown,
whereas the optimal choice corresponds to switch just after $m_2$ leaves $e_2$ (see Figure \ref{fig:test1-3d} below).

These two simple examples show that, in general, the numerical optimization of the traffic light is a very challenging problem,
since there is a wide number of local extrema where the gradient descent algorithm can stop.
To overcome this issue, we perform several runs with random initial guesses for the controls, and we select the solution obtaining the best result.
\begin{figure}[!h]
\centering
\begin{tabular}{cc}
\includegraphics[scale=1]{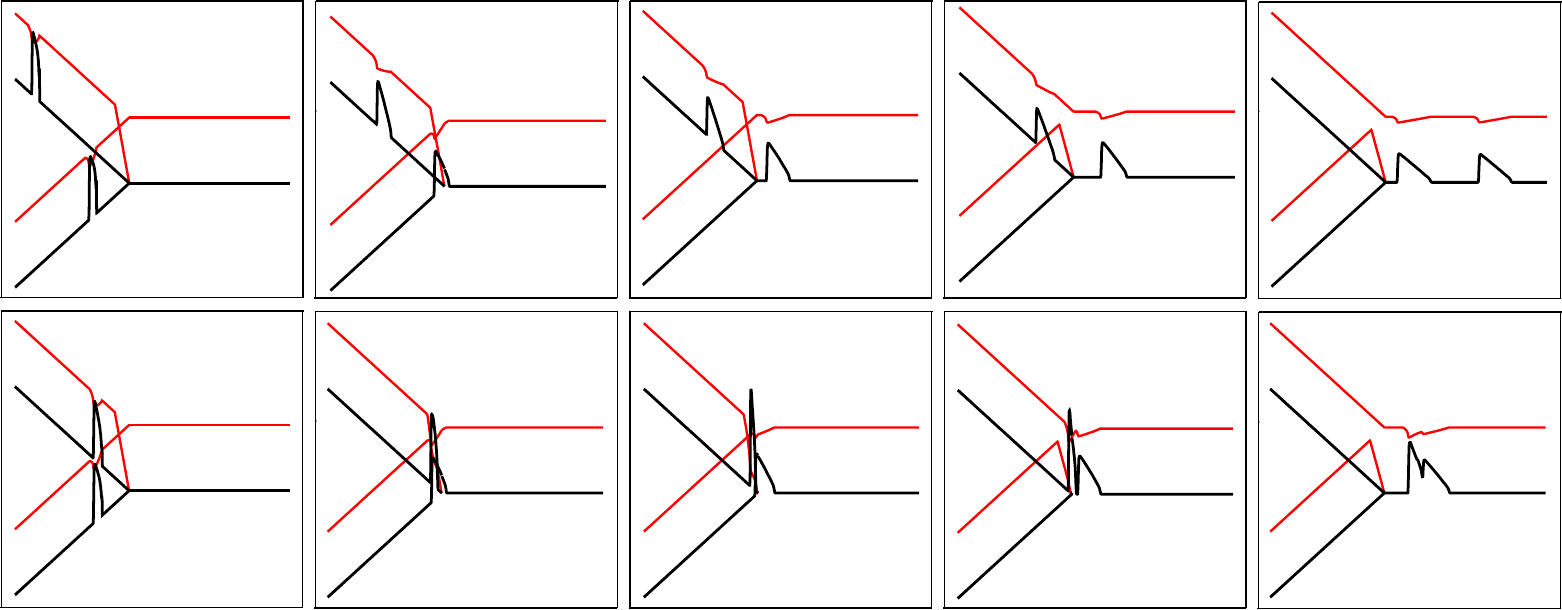}
\end{tabular}
\caption{Optimal solutions for well separated densities (top panels) and overlapping densities (bottom panels)}\label{fig:test1-3d}
\end{figure}
Figure \ref{fig:test1-3d} shows two optimal solutions at different times computed by the gradient descent method,
both achieving the absolute maximum of the corresponding mean velocity. Top panels refer to the case of well separated
densities, bottom panels to the case of overlapping densities. As before, black and red lines represent  $m$ and $v$ respectively.
The fourth frame in each sequence shows the precise
moment of the switch for the traffic light. In the second case we clearly observe that on $e_1$ the traffic is stopped until $m_2$ leaves $e_2$.

We conclude with a more complete example, also including control constraints. All the parameters are the same of the previous tests,
but we fix to $S=5$ the number of switching durations (corresponding to $4$ switching times) and we start with $u_0=0$, i.e. green light on $e_1$.
Moreover, we set the constraints $T^G=0.15$, $T^R=0.3$,
and $m_0$ is given edge-wise by
$$m^1_0(x)=\chi_{[0.1,0.15]}(x)+\chi_{[0.4,0.45]}(x)\,,\quad m^2_0(x)=\chi_{[0.1,0.15]}(x)+\chi_{[0.6,0.65]}(x)\,,\quad m^3_0(x)=0\,.$$
Note that, with this choice, we are mixing together the two cases analyzed before. Indeed, the initial density consists of four blocks which are, respectively,
pairwise overlapped and well separated.
The optimal solution produced by the gradient descent algorithm is $s^*=(0.227,0.251,0.259,0.3,0.21)$. Figure \ref{fig:test2-3d} shows the corresponding evolution at
different times.
\begin{figure}[!h]
\centering
\begin{tabular}{cc}
\includegraphics[scale=1]{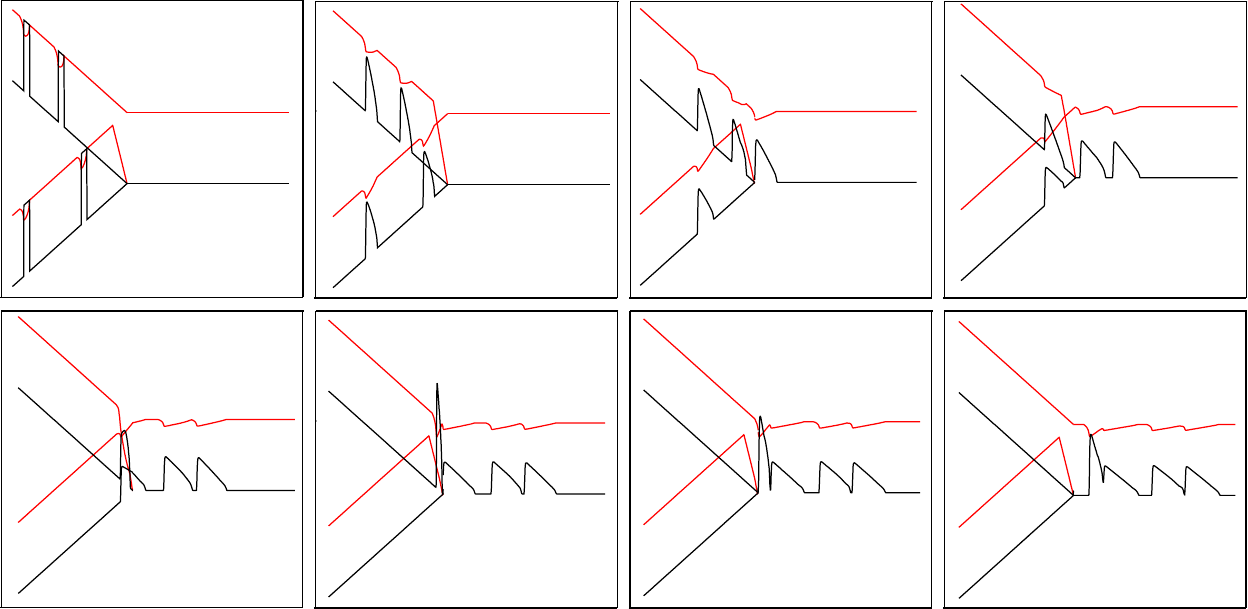}
\end{tabular}
\caption{Optimal solution for a traffic light with $4$ switches}\label{fig:test2-3d}
\end{figure}
We observe that the first switch occurs before $m_2$ approaches the traffic light. This allows the first block of $m_2$ to proceed without slowdowns from $e_2$ to $e_3$.
The second switch occurs immediately after this block leaves $e_2$, so that also the first block of $m_1$ can leave $e_1$ almost undisturbed before
the traffic light switches
again. Now, the remaining densities on $e_1$ and $e_2$ are in overlapping configuration, $m_2$ goes first, while $m_1$ stops. Finally, the last switch occurs just after $m_2$ leaves $e_2$, so that also $m_1$ can move to $e_3$ for the remaining time.


\appendix
\section{Some complementary results for the variational problems}
\begin{proof}[Proof of Lemma \ref{comp1}]
Assume without loss of generality that $M=1$. It is well known that for $\mu \in \cM^+_M(\G_T), |\mu|_{TV} = \mu(\G_T) \leq 1.$
\\By Banach-Alaoglu Theorem it follows the compactness with respect to the weak*-convergence, which implies the same property with respect to the $\|\cdot\|_{BL}^*$ convergence.
\end{proof}

\begin{proof}[Lemma \ref{comp2}]
Since \eqref{hyp_light2} is just a condition which defines the dependence among the components of $u \in \cU$, we prove the compactness of
$$\cU = \{ u \in BV([0,T], \{0,1\}) \,\text{and}\,u\, \text{satisfies \eqref{hyp_light1},}\}.$$
Let $(u_n)_{n \in \N} \subset \cU$.
Denote by $\tau^n_i$ the switching times of $u_n$. By \eqref{hyp_light1},  for every two consecutive switching times $\tau^n_k, \tau^n_{k+1} \in [0,T]$, if $u^n(\tau^n_k) = 1$, then
$$|\tau^n_k - \tau^n_{k+1}| < T^R,$$
otherwise,
$$|\tau^n_k - \tau^n_{k+1}|> T^G.$$
Since $u_{n}(t) \in \{0,1\}$, we can assume that there exists a subsequence, still denoted by $u_n$, such that either $u_n(0)=1$  or $u_n(0)=0$  for every $n \in \N$. Assume now that, w.l.o.g., $u_n(0) = 1$ for every $n \in \N$ and  denote by $I_{n}$ the set of switching times of $u_n$.
It follows that
$$\frac{T}{T^R} \leq \#(I_n)\leq \frac{T}{T^G}.$$
As before, we can assume, w.l.o.g., that that there exists  $N\in\N$ such that $\#(I_n) = N$ for all $n\in\N$.
Since $I_{n} \subset [0,T]$, applying the Cantor diagonal procedure, it follows that there exists a subsequence $(I_{n_k})_{k \in \N}$ such that $\tau_{i}^{n_k} \to \tau_i$ for $i=1, \hdots, N$. In this way, we define a  candidate $u$ as limit for the subsequence $u_{n_k}$ from the switching times set $\{\tau_1, \hdots, \tau_N\}$ and $u(0) = 1$.
To conclude, we only need to show that $u_{n_k} \to u$ in $L^1$. By construction,
$$\|u_{n_k} - u\|_{L^1} = \sum_{i=1}^{N} |\tau_{i}^{n_k} - \tau_{i}| \leq N \sup_{i=1, \hdots, N} |\tau_{i}^{n_k} - \tau_{i}| \to_{k \to \infty} 0$$
\end{proof}
\begin{proof}[Proof of Lemma \ref{closed} (traffic lights)]
In this case, the distirbution $\mu$ has no role since it depends exclusively on $u$. Hence, we reduce on $\cX = \cM^+(\G_T)\times \cU,$ where $\cU$ defined by \eqref{control_light}.\\
 Let $(m_n, u_n)_{n \in \N} \subset A$ such that $(m_n, u_n) \to (m, u)$  with respect the norm $\|\cdot\|_{BL}^* + \|\cdot\|_{L^1}.$\\
The closure on the first component derives from the proof of Lemma 4.1 in \cite{Bon_Butt} and the results in \cite{CDTnonloc}.
\\
Instead, the closure on the second component derives from the compactness of $\cU$.
Indeed, there exists a subsequence $(u_{n_k})_{k \in \N}$ which converges to $\tilde u \in \cU$, but it also converges to $u$ by assumption.
Then, it follows that $u = \tilde u \in \cU.$\\

\end{proof}
\begin{proof}[Proof of Lemma \ref{closed_2} (autonomous cars)]
It follows adopting the argument in the previous proof, for $\cX = \cM^+_M(\G_T) \times \cM^+_M(\G_T) \times \cU$ endowed with 
 the norm $\|\cdot\|_{BL}^* + \|\cdot \|_{BL}^* + \|\cdot\|_{\infty}.$\\

\end{proof}

\bibliographystyle{amsplain}

\end{document}